\documentclass[11pt]{amsart}

\usepackage{amsmath, amssymb}

\ifx\pdfoutput\undefined
\usepackage{graphicx}
\else
\usepackage[pdftex]{graphicx}
\usepackage{epstopdf}
\fi

\usepackage{verbatim, manfnt, amscd, amsthm,enumerate,import, comment}
\usepackage{tikz}

\newtheorem{theorem}{Theorem}[section]
\newtheorem{proposition}[theorem]{Proposition}
\newtheorem{lemma}[theorem]{Lemma}

\newtheorem*{theorem*}{Theorem}

\theoremstyle{definition}

\theoremstyle{remark}
\newtheorem{remark}[theorem]{Remark}

\numberwithin{equation}{section}

\newcommand{\frakg}{{\mathfrak g}}

\newcommand{\bel}{\begin{equation}\label}
\newcommand{\ee}{\end{equation}}

\makeatletter
\newcommand{\tpitchfork}{%
  \vbox{
    \baselineskip\z@skip
    \lineskip-.52ex
    \lineskiplimit\maxdimen
    \m@th
    \ialign{##\crcr\hidewidth\smash{$-$}\hidewidth\crcr$\pitchfork$\crcr}
  }%
}
\makeatother

\begin{document}

\openup1pt


\title
{Inverse spectral results for non-abelian group actions}

\author{Victor Guillemin}\address{Department of Mathematics, Massachusetts Institute of Technology, Cambridge, MA 02139, USA}\email{vwg@math.mit.edu}
\author{Zuoqin Wang}\address{School of Mathematical Sciences, University of Science and Technology of China, Hefei, Anhui 230026, P.R.China}\email{wangzuoq@ustc.edu.cn}

\begin{abstract}
In this paper we will extend to non-abelian groups inverse spectral results, proved by us in an earlier paper, for compact abelian groups, i.e. tori. More precisely,
Let $\mathsf G$ be a compact Lie group acting isometrically on a compact Riemannian manifold $X$. We will show that for the Schr\"odinger operator $-\hbar^2 \Delta+V$ with $V \in C^\infty(X)^{\mathsf  G}$, the potential function $V$ is, in some interesting examples, determined by the $\mathsf G$-equivariant spectrum. The key ingredient in this proof is a generalized Legendrian relation between the Lagrangian manifolds $\mathrm{Graph}(dV)$ and $\mathrm{Graph}(dF)$, where $F$ is a spectral invariant defined on an open subset of the positive Weyl chamber.
\end{abstract}

\thanks{Z. W. is supported by NSFC 11926313, 11721101 and 11571331}
\maketitle

\section{Introduction}

Let $\mathsf G$ be a compact connected Lie group and $\tau: \mathsf G \times X \to X$ an action of $\mathsf G$ on a compact manifold $X$. We will be concerned in this paper with spectral properties of $\mathsf G$-equivariant pseudo-differential operator, $P: C^\infty(X) \to C^\infty(X)$. More explicitly if $P$ is a semi-classical elliptic pseudo-differential operator of order zero and is self-adjoint, then for $f \in C_0^\infty(\mathbb R)$, $f(P)$ is a well-defined smoothing operator and we will be concerned with the spectral invariants
\begin{equation}\label{equiSpecInv}
m(f, g) = \mathrm{trace}\ \tau_g^*f(P).
\end{equation}
To analyze these invariants let $\mathsf{T}$ be a Cartan subgroup of $\mathsf G$ , $\mathfrak t$ its Lie algebra, $\mathfrak t_+^*$ the positive Weyl chamber in $\mathfrak t^*$ and $\mu \in \mathfrak t_+^*$ an integral weight. Then the representation of $\mathsf G$ on $L^2(X)$ decompose into subspaces $L^2(X)_\mu$, where the representation of $\mathsf G$ on $L^2(X)_\mu$ is the sum of irreducible representations of $\mathsf G$ with highest weight $\mu$. Let $\chi_\mu(g)$ be the character of this representation. Then the spectral invariants (\ref{equiSpecInv}) can be read off from the spectral measures
\begin{equation}\label{equivSpecMea}
m_\mu(f) = \mathrm{trace}\ \int \tau_g^* f(P)\chi_\mu(g)dg
\end{equation}
and we will be concerned below with the asymptotic properties of these measures: i.e. for $\hbar=\frac 1k$, $k$ being a positive integer, the asymptotic of
\begin{equation}\label{equivSpecMeak}
m_{k\mu}(f) = \mathrm{trace}\ \int \tau_g^* f(P_\hbar)\chi_{k\mu}(g)dg
\end{equation}
as $\hbar \to 0$ where $\hbar \mapsto P_\hbar$ is the dependence of the semi-classical operator $P$ on $\hbar$.

More explicitly the action, $\tau$, of $\mathsf G$ on $X$ lifts to a Hamiltonian action of $\mathsf G$ on $T^*X$ and modulo assumptions on $\mu$ and $\tau$ (which we will spell out in \S 2) the symplectic reduction, $(T^*X)_\mu$, is well-defined. Moreover, from the symbol, $\mathsf p$, of $P$ one gets a reduced symbol
\begin{equation}\label{pmu}
\mathsf p_\mu: (T^*X)_\mu \to \mathbb R
\end{equation}
and the asymptotic properties of the measure (\ref{equivSpecMeak}) that we will be concerned with below are consequences of the following
\begin{theorem}
The spectral measure $m_{k\mu}$ has an asymptotic expansion in $\hbar$,
\begin{equation}\label{mainasymp}
m_{k\mu}(f) \sim (2\pi \hbar)^{-m} \sum c_i(f, \mu)\hbar^i,
\end{equation}
where $m=\dim X- \frac 12(\dim \mathsf{T} + \dim \mathsf G)$. Moreover,
\begin{equation}\label{firstInv}
c_0(f, \mu) = \int_{(T^*X)_\mu} f(\mathsf p_\mu) \nu_\mu,
\end{equation}
where $\nu_\mu$ is the symplectic volume form on $(T^*X)_\mu$.
\end{theorem}

(For the proof of this for $\mathsf G=\mathsf T$ see \cite{DGS}, and for arbitrary compact Lie group $\mathsf G$ see \cite{GS2}, \S 12.13. For the case where there is no group action, i.e. $\mathsf G=\{1\}$, see \cite{GW}.)

To extract spectral information from this result we will need concrete descriptions of $(T^*X)_\mu$ and $\mathsf p_\mu$, and we will deal with this issue in \S 2 below. Then in the remainder of this  paper we will assume that $X$ is equipped with a Riemannian metric and that $P$ is the semi-classical Schr\"odinger operator, $\hbar^2\Delta+V$, and we will attempt to extract information about $V$ from the spectral data (\ref{firstInv}). For instance for each $\mu \in \mathrm{Int}(\mathfrak t_+^*)$ the spectral measure (\ref{firstInv}) is supported on an interval $[F(\mu), +\infty)$ and we will show in \S 3 and \S 4 that there is a ``generalized Legendre relation" $\Gamma$ associating sets of points in the Lagrangian submanifold
\begin{equation}\label{graphdV}
\mathrm{graph}\ dV
\end{equation}
of $T^*(X_1/\mathsf G)$ ($X_1$ being the open subset of $X$ on which $\mathsf G$ acts locally free) with sets of points in the Lagrangian submanifold
\begin{equation}\label{graphdF}
\mathrm{graph}\ dF
\end{equation}
of $T^*\mathcal U$ ($\mathcal U$ being some open subset in the positive Weyl chamber $\mathfrak t_+^*$) and that in some case this is a canonical transformation, in which case $V$ is spectrally determined. For instance if $\mathsf G=\mathsf{T}$ and $X$ is a toric variety this turns out to be the case modulo genericity assumptions on $X$ (see \cite{GW2}) and in \S 5 and \S 6 we will describe some non-abelian analogues of this result.

\section{Reduction}

In this section $\mathsf G$ will be, as above, a compact connected Lie group, $M$ will be a symplectic manifold, and $\mathsf G \times M \to M$ a Hamiltonian action of $\mathsf G$ with moment map $\phi: M \to \mathfrak g^*$. For $\mathcal O \subset \mathfrak g^*$ a coadjoint orbit the ``reduction of $M$ with respect to $\mathcal O$" is the quotient space, $\phi^{-1}(\mathcal O)/\mathsf G$, which, modulo assumptions on $\mathcal O$ and $\phi$, is a symplectic manifold: \emph{the symplectic reduction of $M$ with respect to $\mathcal O$.} Before we describe these assumptions, however, we first recall that there is an alternative description of this space. Let $\mathsf{T}$, as in \S 1, be a Cartan subgroup of $\mathsf G$, $\mathfrak t$ its Lie algebra and $\mathfrak t_+^* \subset \mathfrak t^*$ the positive Weyl chamber in $\mathfrak t^*$. If $\mu$ is in $\mathrm{Int} (\mathfrak t_+^*)$ (which we'll assume to be the case from now on) the coadjoint orbit $\mathcal O$ of $\mathsf G$ through $\mu$ is, as a $\mathsf G$-space, just $\mathsf G/\mathsf{T}$, and the symplectic reduction of $M$ with respect to $\mathcal O$ can also be viewed as the quotient, $\phi_{\mathsf T}^{-1}(\mu)/\mathsf{T}$. In this section we'll recall what the space looks like when $M$ is, as in \S 1, the cotangent bundle $T^*X$ of a $\mathsf G$-manifold $X$. \footnote{A good reference for the material below is \cite{AM}, \S 4.3-4.5.} From the action of $\mathsf G$ on $X$ one gets an infinitesimal action of the Lie algebra, $\mathfrak g$, on $X$,
\begin{equation} \label{vgvx}
v \in \mathfrak g \mapsto v_X \in \mathrm{Vect}(X)
\end{equation}
and hence for each $p \in X$ a linear map
\begin{equation}
v \in \mathfrak g \mapsto v_X(p) \in T_pX
\end{equation}
which one can dualize to get a linear map
\begin{equation}\label{mom_res}
\phi_p: T_p^*X \to \mathfrak g^*,
\end{equation}
and this linear map is just the restriction to the fiber, $T_p^*X$, of the $\mathsf G$-moment map
\begin{equation}\label{Gmom_map}
\phi: T^*X \to \mathfrak g^*.
\end{equation}

Moreover the $\mathsf T$-moment map at $p$ is just the composite
\begin{equation}\label{Tmom_map}
T_p^*X \to \mathfrak g^* \to \mathfrak t^*
\end{equation}
of the mapping (\ref{mom_res}) and the dual of the inclusion map $\iota_\mathsf{T}: \mathfrak t \hookrightarrow \mathfrak g$. In other words, the $\mathsf G$-moment map, $\phi$, and the $\mathsf{T}$-moment map, $\phi_\mathsf{T}$, are related by
\begin{equation}
\phi_\mathsf{T} = \iota_\mathsf{T}^* \circ \phi.
\end{equation}

Now let $\mu$ be an element of $\mathrm{Int}(\mathfrak t_+^*)$. We claim
\begin{proposition}
If for all $(p, \xi) \in \phi^{-1}(\mu)$, the map (\ref{Tmom_map}) is surjective, then $\mu$ is a regular value of $\phi$.
\end{proposition}
\begin{proof}
$\mu \in \mathrm{Int}(\mathfrak t_+^*)$ is a regular value of $\phi$ if and only if, for every $(p, \xi) \in \phi^{-1}(\mu)$ the action of $\mathsf G$ at $(p, \xi)$ is locally free. However since $\mu$ is in $\mathrm{Int}(\mathfrak t_+^*)$ its stabilizer in $\mathsf G$ is equal to its stabilizer in $\mathsf T$; so the action of $\mathsf G$ at $(p, \xi)$ is locally free if and only if the same is true for the action of $\mathsf{T}$, and this is follows from the surjectivity of (\ref{Tmom_map}) (so that by duality, the $T$-action on $X$ is locally free) and the fact that the stabilizer group of the $\mathsf{T}$-action on $T^*X$ at $(p, \xi)$ is contained in  the stabilizer group of the $\mathsf{T}$-action on $X$ at $p$.
\end{proof}

From this proof we also obtain the following criterion for (\ref{Tmom_map}) to be surjective:
\begin{proposition}\label{surj_locfree}
The map (\ref{Tmom_map}) is surjective if and only if the action of $\mathsf T$ at $p$ is locally free.
\end{proposition}
\begin{proof}
Let $\mathsf T_p$ be the stabilizer of $p$ in $\mathsf{T}$ and $\mathfrak t_p$ its Lie algebra. Then $\mathfrak t_p$ is the kernel of the map
\[
v \in \mathfrak t \to v_X(p) \in T_pX
\]
and hence the image of the map (\ref{Tmom_map}) is $\mathfrak t_p^\perp$. Thus $\mathsf T_p$ is a finite subgroup of $\mathsf T$ if and only if (\ref{Tmom_map}) is surjective.
\end{proof}

Since $X$ is compact there are at most a finite number of subtorus, $\mathsf{T}_r$, which can occur as stabilizers of points of $X$. Thus this result implies
\begin{theorem}\label{surjthm1}
Suppose
\begin{equation}
\mu \not\in \mathfrak t_r^\perp
\end{equation}
for all of these $\mathsf T_r$'s, then for $(p, \xi) \in \phi^{-1}(\mu)$ the map (\ref{Tmom_map}) is surjective.
\end{theorem}
\begin{proof}
If $\mu$ satisfies these conditions, the stabilizer of $p$ has to be a finite subgroup of $\mathsf{T}$ and hence (\ref{Tmom_map}) is surjective.
\end{proof}

Another implication of Proposition \ref{surj_locfree} is
\begin{theorem}\label{surjthm2}
If the action of $\mathsf{T}$ is effective, the map (\ref{Tmom_map}) is surjective for an open dense set of $p$'s.
\end{theorem}
\begin{proof}
If the action of $\mathsf{T}$ is effective then $\mathsf T_p$ is the identity group for an open dense set of $p$'s.
\end{proof}

Henceforth we'll denote by $X_0$ the set of points in $X$ where the stabilizer $\mathsf{T}_p$ is finite, i.e. when the action of $\mathsf{T}$ is locally free.
\begin{theorem}\label{toinf}
Let $p$ be a point in the complement of $X_0$ and $\mu$ an element of $\mathfrak t^*$ satisfying the conditions of Theorem \ref{surjthm1}. Then if $(p_i, \xi_i)$ is in $\phi^{-1}(\mu)$ and $p_i \to p$, $(p_i, \xi_i)$ tends to infinity in $T^*X$.
\end{theorem}
\begin{proof}
If not one can, by passing to a subsequence assume that $(p_i, \xi_i)$ converges in $T^*X$ to a limit point, $(p, \xi)$, and hence that $(p, \xi)$ is in $\phi^{-1}(\mu)$.
\end{proof}

Equipping $T^*X$ with a $\mathsf G$-invariant inner product, $\langle \cdot, \cdot \rangle$, we get a splitting of $T^*X_0$ into a direct sum of vector bundles
\[
T^*X_0 = H \oplus V,
\]
where for each $p \in X_0$, $H_p$ is the kernel of the map (\ref{Tmom_map}) and $V_p$ its ortho-complement with respect to $\langle \cdot, \cdot\rangle_p$. Hence at every point $p \in X_0$ there is a unique element
\begin{equation}\label{alphamup}
\alpha_\mu(p) \in V_p
\end{equation}
such that
\begin{equation}\label{dphiTp}
(\phi_{\mathsf T})_p (\alpha_\mu(p)) = \mu
\end{equation}
and hence a unique $C^\infty$ one-form, $\alpha_\mu$, on $X_0$ with the properties (\ref{alphamup}) and (\ref{dphiTp}). In particular by property (\ref{dphiTp}) the map
\begin{equation}\label{Hphi}
H \to \phi_{\mathsf T}^{-1}(\mu)
\end{equation}
mapping $(p, \xi)$ onto $(p, \xi+\alpha_\mu(p))$ is a $\mathsf T$-equivariant diffeomorphism of $H$ onto $\phi_{\mathsf T}^{-1}(\mu)$ and hence since the action of $\mathsf T$ on $\phi_{\mathsf T}^{-1}(\mu)$ is locally free we get a diffeomorphism of orbifolds,
\begin{equation}
H/\mathsf{T} \to \phi_{\mathsf T}^{-1}(\mu)/\mathsf{T}
\end{equation}
where the orbifold on the right is the \emph{symplectic reduction} of $T^*X_0$ at $\mu$ with respect to the action of $\mathsf T$.

Next note that since $H_p$ is the kernel of the linear map (\ref{Tmom_map}) it is the space of $\xi \in T_p^*X$ satisfying
\begin{equation}
\langle v_X(p), \xi\rangle = 0
\end{equation}
for all $v \in \mathfrak t$. In other words it is the set of all vectors $\xi \in T_p^*X_0$ orthogonal to the orbit of $\mathsf T$ through $p$, or alternatively
\begin{equation}
H=\pi^*(T^*(X_0/\mathsf{T}))
\end{equation}
where $\pi$ is the projection of $X_0$ onto $X_0/\mathsf T$. Hence the $\mathsf{T}$ equivariant diffeomorphism (\ref{Hphi}) gives one a diffeomorphism
\begin{equation}\label{co=sr}
T^*(X_0/\mathsf{T}) \to \phi_\mathsf{T}^{-1}(\mu)/\mathsf{T}
\end{equation}
of the cotangent bundle of $X_0/\mathsf{T}$ onto the symplectic reduction of $T^*X$ at $\mu$ with respect to the action of $\mathsf{T}$ on $T^*X$.

\emph{A cautionary remark}:  The action of $\mathsf T$ on $X_0$ is locally free but not necessarily free; hence this is a diffeomorphism of orbifolds. (However in most of the examples we'll be discussing below these orbifolds are manifolds.)

Next note that by (\ref{Tmom_map}) we have an inclusion
\begin{equation}
\phi^{-1}(\mu) \to \phi_{\mathsf T}^{-1}(\mu)
\end{equation}
and hence an embedding
\begin{equation}
\phi^{-1}(\mu)/\mathsf{T} \to \phi_{\mathsf T}^{-1}(\mu)/\mathsf{T}.
\end{equation}
Thus from the identification (\ref{co=sr}) one gets an embedding
\begin{equation}
\phi^{-1}(\mu)/\mathsf{T} \to T^*(X_0/\mathsf{T}).
\end{equation}
The image of this embedding is a bit complicated to describe at arbitrary points of $X_0/\mathsf{T}$, however it turns out to have a rather simple description over the open subset, $X_1/\mathsf{T}$, where $X_1$ is the set of points, $p \in X_0$, at which the action of $\mathsf G$ itself is locally free. To see this note that if $p$ is in $X_1$, the map
\begin{equation}\label{dphip}
\phi_p: T_p^*X \to \mathfrak g^*
\end{equation}
is surjective so there exists a unique $\alpha_\mu(p) \in T_p^*X$ which is perpendicular to the kernel
\begin{equation}\label{Kp}
K_p = \mathrm{ker}(\phi_p)
\end{equation}
 and $\phi_p$ maps $\alpha_\mu(p)$ onto $\mu$. However, the kernel $K_p$ is contained in the kernel of the map (\ref{Tmom_map}), so this ``$\mathsf G$-equivariant definition" of $\alpha_\mu$ coincides with the ``$\mathsf{T}$-equivariant definition" that we gave above. Moreover the assignment
\begin{equation}
p \in X_1 \to K_p \subset T_p^*X_1
\end{equation}
defines a vector sub-bundle $K$ of $T^*X_1$ sitting inside the horizontal bundle $H|_{X_1}$ and the pre-image, $\phi^{-1}(\mu)$, is, over $X_1$, just the image of $K$ with respect to the mapping (\ref{Hphi}), i.e. over $X_1$, $\phi^{-1}(\mu)$ is a fiber bundle with fiber
\begin{equation}
K_p + \alpha_\mu(p)
\end{equation}
at $p \in X_1$. Moreover over the subspace $X_1/\mathsf{T}$ of $X_0/\mathsf{T}$, $\phi^{-1}(\mu)/\mathsf{T}$ has an equally nice description. Since $\mathsf G$ acts in a locally free fashion on $X_1$, $X_1/\mathsf G$ is well defined as an orbifold, so one has a fibration of orbifold
\begin{equation}
\gamma: X_1/\mathsf{T} \to X_1/\mathsf G,
\end{equation}
and it is easy to see that under the identification
\[
\phi_T^{-1}(\mu)/\mathsf{T} \to T^*X_0/\mathsf{T}
\]
the space
\[
(\phi^{-1}(\mu) \cap T^*X_1)/\mathsf{T}
\]
gets mapped on the ``horizontal" sub-bundle of $T^*(X_1/\mathsf{T})$ with respect to the fibration, $\gamma: X_1/\mathsf{T} \to X_1/\mathsf G$, i.e.
\begin{equation}
(\phi^{-1}(\mu) \cap T^*X_1)/\mathsf{T} \simeq \gamma^* T^*(X_1/\mathsf G).
\end{equation}

We will conclude this section by saying a few words about the elliptic operator, $P$, in section one and its ``reduced symbol", (\ref{pmu}).
\begin{proposition}\label{pproper}
Let
\begin{equation}\label{symbolp}
\mathsf p: T^*X \to \mathbb R
\end{equation}
be the symbol of the operator $P$. Then $p|_{\phi^{-1}(\mu)}$ is proper.
\end{proposition}
\begin{proof}
By ellipticity (\ref{symbolp}) is proper, therefore if its restriction to $\phi^{-1}(\mu)$ were not proper there would exist a sequence of points $(p_i, \xi_i) \in  \phi^{-1}(\mu)$ converging to a point $(p, \xi)$ not on $\phi^{-1}(\mu)$ and this can't happen by Theorem \ref{toinf}.
\end{proof}

Since the function $\mathsf p|_{\phi^{-1}(\mu)}$ is $\mathsf{T}$-invariant there is a unique function
\begin{equation}
\mathsf p_\mu: \phi^{-1}(\mu)/\mathsf{T} \to \mathbb R
\end{equation}
whose pull back to $\phi^{-1}(\mu)$ is $\mathsf p|_{\phi^{-1}(\mu)}$ and this, by definition, is the reduced symbol (\ref{pmu}) of $P$. Thus we get as a corollary of Proposition \ref{pproper}
\begin{theorem}
The reduced symbol (\ref{pmu}) is proper and in particular the spectral invariants (\ref{firstInv}) are well-defined.
\end{theorem}

In addtion we get as a corollary
\begin{proposition}\label{alphaproper}
The one form $\alpha_\mu$, viewed as a map
\begin{equation}
\alpha_\mu: X_0 \to T^*X_0,
\end{equation}
is proper.
\end{proposition}

\section{The Schr\"odinger operator}

As in \S 2 we will equip $T^*X$ with a $\mathsf G$-invariant inner product. Now, however, we will use the inner product to define a $\mathsf G$-invariant Riemannian metric on $X$ and denote by
\begin{equation}
\Delta: C^\infty(X) \to C^\infty(X)
\end{equation}
the associated Laplacian. In addition, given a potential function, $V: X \to \mathbb R$, we get from $\Delta$ and $V$ a semi-classical Schr\"odinger operator
\begin{equation}
\hbar^2 \Delta +V
\end{equation}
which is elliptic, self-adjoint and, thanks to the factor $\hbar^2$, is semi-classically a differential operator of order zero with leading symbol
\begin{equation}\label{pxxi}
\mathsf p(x, \xi) = \langle \xi, \xi\rangle_x +V(x).
\end{equation}

Next recall that for $\mu \in \mathrm{Int}(\mathfrak t_+^*)$ and $p \in X_0$ the level set, $\phi^{-1}(\mu)$, of the moment map, (\ref{Gmom_map}), intersects $T_p^*X$ in the set
\[
K_p +\alpha_\mu(p),
\]
where $K_p$ is the kernel of the map (\ref{mom_res}) (i.e. is a linear subspace of $T_p^*X$). Hence the minimum value of $\mathsf p|_{T_p^*X}$ is just
\begin{equation}\label{alphaplusV}
\langle \alpha_\mu(p), \alpha_\mu(p) \rangle_p+V(p),
\end{equation}
the function
\begin{equation}
p\in X_0 \mapsto \mathsf p_\mu(p):=\langle \alpha_\mu(p), \alpha_\mu(p) \rangle_p+V(p),
\end{equation}
being the ``effective potential of the Schr\"odinger operator restricted to the space, $L^2_\mu(X)$." (See \cite{AM}, \S 4.5). By Proposition \ref{alphaproper}, this function is proper and tends to $+\infty$ as $p$ tends to the boundary of $X_0$, and hence its minimum value,
\[F(\mu) = \min_{p \in X_0} \mathsf p_\mu(p),\]
is well-defined. Moreover, since the spectral measure, (\ref{firstInv}), is supported on the interval, $[c_\mu, +\infty)$ and $c_\mu$ is, by (\ref{pxxi}), equal to $F(\mu)$, the function, $F$, is a spectral invariant of the Schr\"odinger operator.

Let's now assume that the subset, $X_1$, of $X$ on which the action of $\mathsf G$ is locally free is a dense subset of $X_0$, and denote by $Y$ and $Z$ the quotients, $X_0/\mathsf T$ and $X_1/\mathsf G$, and by $Y_1$ the open dense subset $X_1/\mathsf T$ of $Y$. Then since the first summand of (\ref{alphaplusV}) is $\mathsf T$-invariant, it is the pull-back to $X_0$ of a $C^\infty$ function, $W(y, \mu)$ on $Y_1$, and since the second summand is $\mathsf G$-invariant, it is the pull back to $X_1$ of a function on $Y_1$ of the form, $\gamma^*V_{red}$, where $V_{red}$ is a function on $Z$ and $\gamma$ is the projection of $Y_1$ onto $Z$.

Now let $\mathcal U$ be the subset of $\mu$'s in $\mathrm{Int}(\mathfrak t^*_+)$ having the property that (\ref{alphaplusV}) takes its minimum value in $X_1$. Then, for generic choices of $V_{red}$, $\mathcal U$ is an open dense subset of $\mathrm{Int}(\mathfrak t^*_+)$ and $F(\mu)$ can also be thought of as the minimum value of the function
\begin{equation}\label{yWgammaV}
y \in Y \mapsto W(y, \mu) + \gamma^* V_{red}.
\end{equation}

The question we want to explore in the rest of this paper is to what extend $V_{red}$ is determined by the spectral invariants, $c_\mu$, and hence by the function, $F$. To answer this question we'll begin by exploring some relations between these two functions: suppose the function (\ref{yWgammaV}) has a unique minimum point, $y_0 \in Y_1$ and that in addition, $y_0$ is a non-degenerate minimum. Then there exists a neighborhood, $\mathcal U_0$ of $\mu_0$ in $\mathcal U$ such that for every $\mu \in \mathcal U_0$ the function (\ref{yWgammaV}) has a unique minimum at $y =f(\mu)$, and such that the map, $\mu \in \mathcal U_0 \to f(\mu)$, is smooth. Moreover at $y=f(\mu)$,
\begin{equation}\label{3.7.1}
\frac{\partial}{\partial y} \left( W(y, \mu)+\gamma^* V_{red}(y) \right)=0
\end{equation}
and since $F(\mu)=W(f(\mu), \mu)+\gamma^* V_{red}(f(\mu))$, we have
\begin{equation}\label{3.7.2}
\frac{\partial}{\partial \mu}F(\mu) =\frac{\partial}{\partial y} \left( W+\gamma^* V_{red} \right)(f(\mu), \mu)\frac{\partial f}{\partial \mu} +\frac{\partial}{\partial \mu}W(f(\mu), \mu).
\end{equation}
Therefore (\ref{3.7.1}) and (\ref{3.7.2}) imply the identities
\begin{equation}\label{3.8.1}
\frac{\partial W}{\partial y}(y, \mu)=-\frac{\partial}{\partial y}\gamma^* V_{red}(y)
\end{equation}
and
\begin{equation}\label{3.8.2}
\frac{\partial W}{\partial \mu}(y, \mu)=\frac{\partial}{\partial \mu}F(\mu)
\end{equation}
at $y=f(\mu)$. However, the condition ``$y=f(\mu)$" is already implicitly implied by (\ref{3.8.1}), and $W$ depends neither on $V_{red}$ nor on $F$, so these identities should in principle enable us to read off information about $V_{red}$ from equivariant spectral data and, in particular, give one some hope of actually determining $V_{red}$ if the space $\mathcal U_0 \subset \mathrm{Int}(\mathfrak t^*_+)$ on which $F$ is defined has the same dimension as the space, $X_1/\mathsf G$, on which $V_{red}$ is defined. More explicitly let
\begin{equation}
\Gamma_Z \subset T^*Y \times T^*Z
\end{equation}
be the canonical relation defined by stipulating that $((y, \xi), (z, \eta))$ is in $\Gamma_Z$ if and only if
\begin{equation}
z=\gamma(y) \quad \mbox{and} \quad -\xi=(d\gamma)_y^* \eta.
\end{equation}
Similarly for $\mathcal U_0 \subset \subset \mathrm{Int}(\mathfrak t^*_+)$, let
\begin{equation}
\Gamma_W \subset T^*\mathcal U_0 \times T^*Y
\end{equation}
be the canonical relation defined by stipulating that $((\mu,\tau), (y, \xi))$ is in $\Gamma_W$ if and only if
\begin{equation}
\tau= \frac{\partial W}{\partial \mu} (y, \mu) \quad \mbox{and} \quad \xi=-\frac{\partial W}{\partial y} (y, \mu).
\end{equation}
In addition, suppose that these canonical relations are transversally composable, i.e. that the composite canonical relation
\begin{equation}\label{Gammaiscomp}
\Gamma =\Gamma_Z \circ \Gamma_W
\end{equation}
is well-defined. Then if $Z$ and $\mathfrak t^*$ have the same dimension, i.e. if
\begin{equation}\label{X=G+T}
\dim X = \dim \mathsf G + \dim \mathsf T,
\footnote{If $X$ is a $\mathsf G$-manifold with $\dim X/\mathsf G < \dim \mathsf T$, then one may pick a subtorus $\mathsf T_{\!1}$ of $\mathsf T$ with $\dim \mathsf T_{\!1} =\dim \mathsf T - \dim X/\mathsf G$ and pick a $T_{\!1}$-toric manifold $Y$, so that the natural $\mathsf G \times \mathsf T_{\!1}$-action on $X \times Y$ satisfies the dimension assumption.}
\end{equation}
it makes sense to ask whether the canonical relation $\Gamma$ is actually the graph of a canonical transformation
\begin{equation}
\Phi: T^*\mathcal U_0 \to T^*Z,
\end{equation}
or at least that a local version of this assertion is true in the neighborhood of a point $(\mu_0, z_0)$ in $\mathcal U_0 \times Z$. This would then imply by (\ref{3.8.1}) and (\ref{3.8.2}) that the graph of $-dV_{red}$ is the image of the graph of $dF$ with respect to this transformation and hence that $V$ is spectrally determined up to an additive constant (or that a local version of this assertion is true in a neighborhood of $z_0$). We will examine this question in more detail in the next section. More explicitly, we will describe conditions on the function, $W(y, \mu)$, and the fibration, $Y \to Z$, which guarantee that $\Gamma_Z$ and $\Gamma_W$ are transversally composable and that their composition is the graph of a canonical transformation.

\section{The generalized Legendre transform}

Let $Y$ be an $n+k$ dimensional manifold, $Z$ and $ U$ be $n$ dimensional manifolds and $\pi: Y \to Z$ a fibration. Given $W(y, \mu) \in C^\infty(Y \times \mathcal U)$ we would like to understand the composite canonical relation, $\Gamma_\pi \circ \Gamma_W$, where $\Gamma_\pi \subset T^*Y \times T^*Z$ is the canonical relation
\begin{equation}
(z, \xi, y, \eta) \in \Gamma_\pi \  \mbox{iff} \ z=\pi(y) \ \mbox{and} \ \eta = (d\pi)_y^*\xi
\end{equation}
and $\Gamma_W \subset T^*\mathcal U \times T^*Y$ is the canonical relation
\begin{equation}
(y, \eta, \mu, \nu) \in \Gamma_W\  \mbox{iff} \ \eta=-\frac{\partial W}{\partial y}(y,\mu) \ \mbox{and} \ \nu = \frac{\partial W}{\partial \mu}(y,\mu)
\end{equation}

To do so we will first fix some notation. For every $y \in Y$ let $F_y$ be the fiber of $\pi$ containing $y$ and let $T^*_{vert}Y$ be the vector bundle whose fiber at $y \in Y$ is the cotangent space of $F_y$ at $y$. Given a function $\rho \in C^\infty(Y)$ we will define its fiber derivative, $d_{fiber}\rho$, to be the section of $T^*_{vert}Y$ which at $y \in Y$ takes the value
\begin{equation}
(d\iota_F^*\rho)(y),
\end{equation}
$\iota_F$ being the inclusion map, $F_y \to Y$.

Now let's fix $\mu \in \mathcal U$ and for the moment regard $W(y, \mu)$ as a function $W_\mu(y)$ on $Y$. We will make the assumption

\begin{minipage}{0.1\linewidth}
(I)\end{minipage}
\begin{minipage}{0.8\linewidth}
For every $\mu \in \mathcal U$, the section $d_{fiber}W_\mu$ of $T^*_{vert}Y$ intersects the zero section of $T^*_{vert}Y$ transversally.
\end{minipage}

To see what this condition means in coordinates let $y=(z,v)$ be a $\pi$-adapted coordinate system on $Y$, i.e. let $z=(z_1, \cdots, z_n)$ be a coordinate system on $Z$ and for $z$ fixed, let $v=(v_1, \cdots, v_k)$ be a coordinate system on the fiber above $z$. Then in coordinates it's easy to see that condition (I) reduces to the condition that the $k \times k$ matrix
\begin{equation}\label{Wmatrix}
\frac{\partial^2 W_\mu}{\partial v_i \partial v_j} (z,v), \quad 1 \le i, j \le k
\end{equation}
be non-degenerate at points where $\frac{\partial }{\partial v}W_\mu(z,v)=0$, in other words that at such points the mapping
\begin{equation}
v \mapsto  \frac{\partial  }{\partial v}W_\mu(z,v)
\end{equation}
be, for $z$ fixed, locally a diffeomorphism.

One implication of condition (I) is that the set
\begin{equation}
Z_\mu = \{y \in Y\ |\ d_{fiber}W_\mu(y)=0\}
\end{equation}
is a submanifold of $Y$ of dimension $n$ and that the projection
\begin{equation}\label{4.7}
\pi: Z_\mu \to Z
\end{equation}
is locally a diffeomorphism. We will now strengthen this assumption by assuming

\begin{minipage}{0.1\linewidth}
	(II)\end{minipage}
\begin{minipage}{0.8\linewidth}
	The projection $ \pi : Z_\mu \rightarrow Z $ is a covering map
\end{minipage}

\vspace{5pt}
Thus in particular if $ Z $ is simply connected this condition implies
\vspace{5pt}

\begin{minipage}{0.1\linewidth}
	$ (\mathrm{III}) $\end{minipage}
\begin{minipage}{0.8\linewidth}
	Each connected component of $ Z_\mu $ is mapped diffeomorphically onto $ Z $ by the map \eqref{4.7}
\end{minipage}

We will now show that the conditions (I)--(III) are satisfied by the function (\ref{yWgammaV}), at generic points of $Y_1$. To do so we will begin by giving an alternative description of this function: As in \S 1 let $\mathsf G$ be a compact connected Lie group, $\mathsf G \times X \to X$ an effective action of $\mathsf G$ on $X$, $\langle \cdot, \cdot \rangle$ a $\mathsf G$-invariant inner product on $T^*X$ and $X_1$ the open subset of $X$ on which the action of $\mathsf G$ is locally free. Then for $p \in X_1$ one has an injective linear mapping
\[
v \in \mathfrak g \to v_X(p) \in T_pX
\]
and a dual moment mapping
\begin{equation}\label{413}
\phi_p: T_p^*X  \to \mathfrak g^*.
\end{equation}
For $\mu \in \mathfrak g^*$ let $\alpha_\mu(p)$ be the unique element of $(\mathrm{ker}\phi_p)^\perp$ that gets mapped by (\ref{413}) onto $\mu$ and let $\alpha_\mu \in \Omega^1(X_1)$ be the one form, $p \in X_1 \mapsto \alpha_\mu(p)$. By $\mathsf G$-equivariance the map
\begin{equation}\label{414}
\mu \in \mathfrak g^* \to \alpha_\mu \in \Omega^1(X_1)
\end{equation}
intertwines the coadjoint action of $\mathsf G$ on $\mathfrak g^*$ and the action $g \mapsto \tau_g^*$ of $\mathsf G$ on $\Omega^1(X_1)$.

Consider now the function
\begin{equation}
W(x, \mu)  = \langle \alpha_\mu(x), \alpha_\mu(x)\rangle_x
\end{equation}
(This is a slight variant of the function $W(y, \mu)$ in display (\ref{yWgammaV}) since it is defined on $X_1 \times \mathfrak g^*$ rather than on $X_0/T \times \mathfrak t^*$.) By (\ref{414}) this function has the equivariant property
\begin{equation}\label{Equi416}
W(gx, \mu) = W(x, \mathrm{Ad}(g)^*\mu\rangle
\end{equation}
and we will examine the non-degeneracy of the matrix  (\ref{Wmatrix}) using this more equivariant description of $W$. In (\ref{Wmatrix}) the $z_i$'s are, for a point $p_0 \in X_1$, coordinates on a neighborhood of the image point in $X_1/\mathsf G$ and the $v_i$'s are coordinates on the fiber of the fibration, $Y_1 \to Z_1$, above this point. Note, however, that for $\mu \in \mathrm{Int}(\mathfrak t_+^*)$ the stabilizer of $\mu$ in $\mathsf G$ with respect to the $\mathrm{Ad}^*$ action of $\mathsf G$ on $\mathfrak g^*$ is $\mathsf T$ and the $\mathsf G$ orbit through $\mu$ is just the coadjoint orbit, $\mathcal O$, through $\mu$ in $\mathfrak g^*$.

Hence by the equivariance property \eqref{414} the non-degeneracy condition (I) can be reformulated as follows: Let $ B $ be the quadratic form on $ \frakg^* $ associated with the inner product $ \langle\, , \, \rangle_p $ on $ T_p X $ via the bijective linear map
\[ \alpha_\mu (\phi) \in (\operatorname{ker} \phi_p)^\perp \to \mu  \]
defined by \eqref{414} and let $ \rho = B \vert_{\mathcal O}. $ Then condition (I) is equivalent to ``For every $ \mu \in {\mathcal O}$ at which $ d\rho_\mu = 0,\ \ (d^2 \rho)_\mu$ is non-degenerate.''

We will prove in an appendix to this paper that for generic choices of $B$ this condition is satisfied for all generic coadjoint orbits of $\mathsf G, $ i.e. orbits of the form, $\mathsf G \cdot \mu, \ \mu \in \mathfrak t^*_+ \, . $

Turning to the condition (II) and (III) we note that if we assume that the action of $\mathsf G$ on $X_1$ is free rather than just locally free, $Y_1$ and $Z_1$ are the quotient manifolds $X_1/\mathsf T$ and $X_1/\mathsf G$ and in particular $Y_1$ is a fiber bundle over $Z_1$ with fiber $\mathcal O=\mathsf G/\mathsf T$.

Hence by the equivariance property (\ref{Equi416}) of $W(x, \mu)$, these conditions are satisfied as well.

\section{Inverse results}

Coming back to the canonical relation \eqref{Gammaiscomp} we note that in view of the computations in Section 4 this canonical relation has a finite number of connected components, among them a minimal component which relates the Lagrangian manifolds in $T^*Z$ and $T^*\mathcal U$ defined by the graphs of $dV_{red}$ and $dF(\mu)$. Hence, as we explained in Section 3 this potentially gives us an inverse spectral result that determine $V_{red}$, up to an additive constant, from spectral data of the Schr\"odinger operator $\hbar^2 \Delta + V$. We will now describe some assumptions that will enable us to prove this is the case.

Recall from Section 2 that the canonical relation we just alluded is the canonical relation
\[\Gamma \subset T^*(Z \times \mathcal U)\]
defined by the graph of the map
\begin{equation}
(z, \mu) \in Z \times \mathcal U \to dW_\mu(z)
\end{equation}
and to prove the inverse spectral result described above by the method of Section 3 we will need to show that for generically chosen $\mathsf G$-invariant Riemannian metric on $X$ this is the graph of a symplectomorphism. We have already shown in Section 3 that for a generically chosen $\mathsf G$-invariant Riemannian metric on $X$ conditions (I) and (II) are satisfied and we will show below tat this stronger result is true.

The key ingredient in the proof is the following alternative description of the function $W(y, \mu)$ on $Y \times \mathfrak t_+^*$. Fixing an $x_0 \in X_1$  let $y_0$ and $z_0$ be the projections of $x_0$ onto $Y$ and $Z$. Then via the map $g \mapsto g \cdot x_0$, the fiber of $X_1$ above $z_0$ can be identified with $\mathsf G$ and the fiber above $z_0$ in $Y$ with $\mathsf G/\mathsf T$.  Mover, the fiber above $z_0$ in $Y$ can, via the map
\[
\mathrm{Ad}^*: \mathsf G \times \mathfrak t^* \to \mathfrak g^*, \quad (g, \mu) \mapsto \mathrm{Ad}(g)^* \mu,
\]
be identified with the coadjoint orbit $\mathsf G \cdot \mu = \mathcal O \subset \mathfrak g^*$. Moreover the Riemannian metric on $X$ defines a positive definite bilinear form on $T^*_{x_0}(\pi^{-1}(z_0))$ and hence, via the identification above, a positive definite bilinear form, $B$, on $\mathfrak g^*$, and the restriction of this bilinear form to $\mathcal O$ is, via the identification above, just the function $\langle \alpha_\mu, \alpha_\mu\rangle$ restricted to the fiber in $Y$ above $z_0$. Thus the computation of the critical values of the function
\[
y \in Y \mapsto W(y, \mu)
\]
on the fiber of $Y$ above $z$ can be reduced to the computation of the critical value of this function.

A particularly interesting example of a bilinear function on $\mathfrak g^*$ is the Killing form $C: \mathfrak g^* \to \mathbb R$ which has the property that it is a positive definite quadratic form on $\mathfrak g^*$ and is $\mathsf G$-invariant, i.e., is constant on coadjoint orbits. Hence in the alternative description of the function $W(y, \mu)$ that we have just provided we can, if so minded, replace the $B$ in this alternative definition of $W(y, \mu)$ by $B+\lambda C$, where $\lambda$ is a constant, and, in fact, in a small neighborhood of $z_0$ we can make this constant a function $\rho(z, \mu)$, where $z=\pi(z)$ and $\mu \in \mathfrak t_+^*$.

Let us now assume as in Section 4 that $B$ is a Morse function on the coadjoint orbit
\[
\mathcal O=\mathsf G\cdot \mu, \quad \mu \in \mathfrak t^*_+
\]
having distinct critical values and as above let $\widetilde W(z, \mu)$ be the minimal value of the function, $W(y, \mu)$ defined by this $B$. Then, replacing $B$ by $B+\lambda C$, $\widetilde W(z, \mu)$ gets replaced by $\widetilde W(z, \mu)+f(z, \mu)$, where
\begin{equation}
\label{e5.2}
f(z, \mu) = \rho(z, \mu) C(\mu, \mu).
\end{equation}

Therefore, the implication of this observation is that by perturbing the Riemannian metric that we used to define the Schr\"odinger operator, $\hbar^2\Delta +V$, we can, on a neighborhood of $z_0$, convert the function $\widetilde W(z, \mu)$ to a function
\begin{equation}\label{e5.3}
\widetilde W(z, \mu) +f(z, \mu),
\end{equation}
where $f(z, \mu)$ is  a more or less arbitrary function of $(z, \mu)$ on this neighborhood. In particular we can perturb $\widetilde W(z, \mu)$ so that, locally on this neighborhood the matrix
\begin{equation}\label{e5.4}
\frac{\partial^2\widetilde W(z, \mu)}{\partial z_i\partial \mu_j}, \qquad 1 \le i, j \le n
\end{equation}
is non-degenerate, i.e. locally near $(z_0, \mu)$, the canonical relation
\begin{equation}
\Gamma_W = \mathrm{graph}\ dW \subset T^*(Z \times \mathcal U)
\end{equation}
is the graph of a symplectomorphism
\[\gamma: T^*\mathcal U \to T^*Z\]
mapping the graph of $dF$ onto the graph of $dV_{red}$ (in other words determining $V_{red}$, up to an additive constant, from the spectral data supplied by $F(\mu)$.)

\section{Killing metrics}
We will define a metric on the fiber bundle, $ Y_1 \rightarrow Z_1, $ to be a \emph{Killing metric} if its restriction to the fibers
\[ \mathcal{O}_\mu, \ \mu \in t^\ast_+ \]
of the fibration are of the form, $ \rho (z, \mu) \mathcal{C}_\mu $ where $ \mathcal{C}_\mu $ is the Killing form on $ \mathcal{O}_\mu $ and $ \rho (z, \mu) $ a $ C^\infty $ function on $ Z \times t^\ast_+ $ (i.e. looks like the perturbative term \eqref{e5.2} in the expression \eqref{e5.3}.) For metrics of this form the non-degeneracy condition \eqref{e5.4} reduces to the condition that for all $ (z, \mu) $ the matrix
\begin{equation}\label{e6.1}
\frac{\partial}{\partial z_i} \frac{\partial}{\partial \mu_j} \rho (z, \mu), \quad 1 \leq i,j \leq \mu
\end{equation}
be non-degenerate, and in a projected sequel to this paper we will use the techniques developed above to prove inverse spectral results for some interesting examples of $ \mathsf{G}$-manifolds with the property that they are homogeneous spaces for the actions of a Lie group containing $ \mathsf{G} $.  We will confine ourselves here however to illustrating how this can be done if one takes $ X $ to be $ \mathbb{C}P^2 $ (viewed as a degenerate coadjoint orbit of $ S \mathcal{U}(3) $) and takes $ \mathsf{G}$ to be $ S \mathcal{U} (2). $ To do so we will make use of the following elementary result

\begin{lemma}\label{lemma}
Let $Z_1$ and $Z_2$ be subintervals of the real line, $f_1$ and $f_2$ functions on $Z_1$ and $Z_2$, $\pi_i: T^*Z_i \to Z_i$ the cotangent projection and $\Gamma \subset T^*X_1 \times T^*X_2$ a canonical relation having the property that
\begin{equation}\label{star}
\Gamma_{df_1} =\Gamma \circ \Gamma_{df_2}.
\end{equation}
In addition suppose $\frac{df_1}{dx_1}(x_1)$ and $\frac{df_2}{dx_2}(x_2)$ are strictly positive. Then $\Gamma$ is the graph of a symplectomorphism.
\end{lemma}
\begin{proof}
Without loss of generality we can assume $f_1$ and $f_2$ are the coordinate functions on $X_1$ and $X_2$ and interpret (\ref{star})  to mean that
\[ \pi_1^* df_1 - \pi^*_2 df_2\]
restricted to $\Gamma$ vanishes. Then $f_1-f_2$ is constant on $\Gamma$, so without loss of generality we can assume $f_1=f_2$ on $\Gamma$, i.e. since $f_1$ and $f_2$ are the coordinate functions on $X_1$ and $X_2$, that $X_1=X_2$, and that
\[(x_1, \xi_1, x_2, \xi_2) \in \Gamma  \Longleftrightarrow x_1=x_2 \ \mbox{and} \ \xi_1=dx_1=\xi_2=dx_2,\]
i.e. we can conclude that $\Gamma$ has to be the identity map.
\end{proof}

To apply this result to the action of $ S \mathcal{U} (2) $ on the space $ \mathbb{C}P^2 $ we must first specify what we mean by this action. This is defined by thinking of $ S \mathcal{U} (2) $ as the subgroup of $ S \mathcal{U} (3) $ consisting of linear mappings which fix the vector $ (0,0,1) $ and then taking the action of $ S \mathcal{U} (2) $ on $ \mathbb{C}P^2 $ to be the quotient action on $ (\mathbb{C}^3-0) / \mathbb{C}-0 $. This is not a free action since it fixes the image in $ \mathbb{C}P^2 $ of the vector $ (0,0,1); $ but it does act freely on the complement $ (\mathbb{C}P^2)_1 $ of this point and the quotient by this action is just $ (0, \infty) $. Thus in terms of this notation we have:
\begin{equation}\label{e6.3}
(\mathbb{C}P^2)_1 / S \mathcal{U}(2) = (0, \infty)
\end{equation}
and
\begin{equation}\label{e6.4}
t^\ast_+ = (0, \infty)
\end{equation}
so we can think of these spaces as copies of the interval $ (0,8) $ and the functions
\begin{enumerate}
	\item[(I)] $ f(\mu, z) = \langle \alpha_\mu (z), \alpha_\mu (z) \rangle + V(z) $ \\
	and
	\item[(II)] $ \tilde{f} (\mu, z) = \langle \alpha_\mu (z), \alpha_\mu (z) \rangle  $
\end{enumerate}
\noindent as functions on the product, $ (0, \infty) \times (0, \infty), $ of these intervals.

The function (I) is a bounded perturbation of the function (II) and hence for $ \mu $ large
\[ F(\mu) = \underset{z}{\min} f(\mu, z) \]
is a bounded perturbation of the function
\[ \begin{aligned}
\tilde{F} (\mu) & = \underset{z}{\min} \langle \alpha_\mu(z), \alpha_\mu (z) \rangle \\
& = \mu^2  \underset{z}{\min} \langle \alpha_1 (z), \alpha_1 (z) \rangle \\
& = \tilde{C} \mu^2
\end{aligned} \]
Hence for $ \mu  $ large
\[ \frac{\partial F}{\partial \mu } (\mu) \neq 0 \, . \]

We will now assume by hypothesis that the potential function, $ V_{red} (z), $ is strictly increasing and hence by Lemma \ref{lemma} that the canonical relation defined by \eqref{Gammaiscomp} is the graph of a symplectomorphism, i.e. that the spectral invariant, $ F(\mu) $ determines the potential function $ V_{red} (z). $

\begin{remark}
The argument above applies, mutatis mutandis, to lots of other examples besides the $ \mathbb{CP}^2 $ example above, i.e. to all examples for which the space, $ Z = X/G, $ is one dimensional. (Another interesting example of such a space is the generic coadjoint orbit of $ SO(4) $ viewed as an $ SO(3) $ manifold.)
\end{remark}

\appendix

\section{The genericity of admissible metrics}

We will prove below the genericity result for quadratic forms, $ B, $ that we cited in \S 4.

Let $\mathcal B$ be the set of all positive definite quadratic forms on $\mathfrak g^*$. We are interested in $B \in \mathcal B$ such that
\begin{center}
\begin{minipage}{0.05\linewidth}
$\mathrm{(*)}$
\end{minipage}
\begin{minipage}{0.8\linewidth}
for every coadjoint orbit $\mathcal O_\mu = \mathsf G\cdot \mu$, $\mu \in \mathfrak t_+^*$, the function
\begin{equation*}
\rho_B^\mu=B|_{\mathcal O_\mu}: \mathcal O_\mu \to \mathbb R
\end{equation*}
is a Morse function.
\end{minipage}
\end{center}

\begin{theorem}
The set of positive definite quadratic functions $B: \mathfrak g^* \to \mathbb R$ such that $\mathrm{(*)}$ holds is dense and open in $\mathcal B$.
\end{theorem}
\begin{proof}
Let $Y \to \mathfrak t_+^*$ be the fiber bundle over $\mathfrak t_+^*$ with fiber $\mathcal O_\mu$ over $\mu$, and let $W$ be the fiber bundle over $\mathfrak t_+^*$ with fiber $T^*\mathcal O_\mu$ over $\mu$. Then, via the map $T^*\mathcal O_\mu \to \mathcal O_\mu$, $W$ can also be thought of as a fiber bundle over $Y$. Moreover the map
\[\iota: Y \to W\]
mapping $\mathcal O_\mu$ onto the zero section of $T^*\mathcal O_\mu$ gives one an embedding of $Y$ into $W$.
We define a map $\tilde \rho: Y \times \mathcal B \to W$ as follows: for any $y=(\mu, z) \in Y$, where $z \in \mathcal O_\mu$, we let
\[
\tilde \rho (y, B) = (\mu, z, (d\rho_B^\mu)_z) \in W.
\]
It is easy to see that $\tilde \rho$ intersects the embedded image of $Y$ in $W$ transversally, i.e. for any $(\mu, z, B)$ such that $(d\rho_B^\mu)_z=0$, one has
\[
\mathrm{Im}(d\tilde \rho)_{\mu, z, B} + T_{\mu, z, 0}Y=T_{\mu, z, 0}W,
\]
which follows from the fact that for any given $(\mu, z) \in Y$, the map
\[
\phi_{\mu, z}: \mathcal B \to T^*_z\mathcal O_\mu, \quad B \mapsto (d\rho_B^\mu)_z
\]
is always surjective.

 Hence by the transversality theorem of Thom, the set of
 $B \in \mathcal B$ such that the map
\[\rho_B: Y \to   Y \times \mathcal B \to W\]
is transverse to the embedded image, $\iota(Y)$ of $Y$ in $W$, is dense and open. The latter implies that the map
\[d\rho_B^\mu: \mathcal O_\mu \to T^*\mathcal O_\mu\]
intersects the zero section of $T^*\mathcal O_\mu$ transversally, in other words,  $\rho_B^\mu|_{\mathcal O_\mu}$ is a Morse function on $\mathcal O_\mu$ for all $\mu \in \mathfrak t_+^*$.
\end{proof}



\begin{thebibliography}{99}

\bibitem[AM]{AM}
{R. Abraham and J. Marsden},
{\sl Foundations of Mechanics},
{Westview Press, 1994.}


\bibitem[DGS]{DGS}
{E. Dryden, V. Guillemin and R. Sena-Dias},
{Semi-classical weights and equivariant spectral theory},
{arXiv: 1401.8285.}

\bibitem[GS]{GS2}
{V. Guillemin and S. Sternberg},
{\sl Semi-classical Analysis},
{International Press, 2013.}

\bibitem[GW]{GW}
{V. Guillemin and Z. Wang},
{Semiclassical spectral invariants for Schr\"odinger operators },
{\sl Jour. Diff. Geom.} 91 (2012), 103-128.

\bibitem[GW2]{GW2}
{V. Guillemin and Z. Wang},
{The generalized Legendre transform and its applications to inverse spectral problems},
{\sl Inverse Problems} 32 (2016), 015001 (22pp).

\end{thebibliography}
\end{document}